\documentclass[12pt,a4paper,oneside]{article}

\usepackage[utf8]{inputenc}
\usepackage[english]{babel}
\pdfoutput=1

\usepackage{amsmath,amsfonts,amssymb,amsopn,amscd,amsthm,mathtools}
\allowdisplaybreaks

\usepackage[yyyymmdd,hhmmss]{datetime}

\usepackage{fancyhdr}
\newcommand{\CompleteSetsOf}[1]{\mathcal{C}_{#1}}
\newcommand{\CompleteSetsOfG}{\CompleteSetsOf{G}}

\newcommand{\CliquesOf}[1]{\mathcal{K}_{#1}}
\newcommand{\CliquesOfG}{\CliquesOf{G}}

\newcommand{\CliqueGraphOf}[1]{\mathbf{K}_{#1}}
\newcommand{\CliqueGraphOfG}{\CliqueGraphOf{G}}

\newcommand{\CliqueTreesOf}[1]{\mathcal{C}\!\mathcal{T}_{#1}}
\newcommand{\CliqueTreesOfG}{\CliqueTreesOf{G}}

\newcommand{\ReducedCliqueGraphOf}[1]{\mathbf{R}_{#1}}
\newcommand{\ReducedCliqueGraphOfG}{\ReducedCliqueGraphOf{G}}

\newcommand{\CFGeneratedBy}[1]{{F(#1)}}
\newcommand{\MaxGeneratorOf}[1]{C(#1)}

\newcommand{\CliqueFamiliesOf}[1]{\mathcal{F}_{#1}}
\newcommand{\CliqueFamiliesOfG}{\CliqueFamiliesOf{G}}

\newcommand{\StrictSubfamilyGraphOf}[1]{\Gamma_{#1}}
\newcommand{\StrictSubfamilyGraphOfCF}{\StrictSubfamilyGraphOf{F}}

\newcommand{\MaxSubfamiliesOf}[1]{\mathcal{M}_{#1}}
\newcommand{\MaxSubfamiliesOfF}{\MaxSubfamiliesOf{F}}

\newcommand{\CFEquivalenceRelation}{\sim_{F}}
\newcommand{\CFEquivalenceClassOf}[1]{[#1]_{\CFEquivalenceRelation}}

\newcommand{\SameFamilyGraphOf}[1]{\Xi_{#1}}
\newcommand{\SameFamilyGraphOfCF}{\SameFamilyGraphOf{F}}

\newcommand{\ContractionGraphOf}[1]{\Delta_{#1}}
\newcommand{\ContractionGraphOfCF}{\ContractionGraphOf{F}}

\newcommand{\NN}{\mathbb{N}} 

\newcommand{\Iff}{\Leftrightarrow}
\newcommand{\Then}{\Rightarrow}

\DeclareMathOperator{\Argmax}{argmax}

\newcommand{\Set}[1]{\{#1\}}

\newcommand{\Cardinality}[1]{\vert #1 \vert}

\newcommand{\InducedBy}[1]{[#1]}
\newcommand{\ContractedBy}[1]{/#1}
\newcommand{\SpanningTreesOf}[1]{\mathcal{T}_{#1}}

\newcommand{\ModuloRDistance}[1]{{||#1||_R}}

\newtheorem{theorem}{Theorem}[section]
\newtheorem{lemma}[theorem]{Lemma}
\newtheorem{proposition}[theorem]{Proposition}
\newtheorem{corollary}[theorem]{Corollary}

\theoremstyle{remark}

\newtheorem{definition}[theorem]{Definition}

\newcommand{\Abstract}{
\begin{abstract}
We investigate clique trees of infinite locally finite chordal graphs.
Our main contribution is a bijection between the set of clique trees and the product of local finite families of finite trees.
Even more, the edges of a clique tree are in bijection with the edges of the corresponding collection of finite trees.
This allows us to enumerate the clique trees of a chordal graph and extend various classic characterisations of clique trees to the infinite setting.
\end{abstract}
}

\newcommand{\TitleFull}{Clique trees of infinite locally finite chordal graphs}
\newcommand{\TitleShort}{Infinite clique trees}
\newcommand{\TitlePDF}{Infinite\ clique\ trees}

\newcommand{\AuthorsFull}{
 Christoph Hofer-Temmel
 \;and\;
 Florian Lehner
}

\newcommand{\AuthorsShort}{Hofer-Temmel \& Lehner}

\newcommand{\Keywords}{Keywords:
chordal graph,
clique tree,
minimal separator}

\newcommand{\Head}{
 \maketitle
 \Abstract{}
 \Keywords{}\\\MSC{}
}

\usepackage[pdftex%
 ,bookmarks=true%
 ,bookmarksopen=true%
 ,bookmarksopenlevel=5%
 ,colorlinks=false%
 ,pdftitle=Arxiv\ \pdfdate:\ \TitlePDF%
]{hyperref}

\usepackage[numbers,square]{natbib}

\title{\TitleFull}
\author{\AuthorsFull}
\date{}
\begin{document}

\Head{}
\tableofcontents

\section{Introduction}
\label{sec_intro}
\par
A graph is chordal, if for every cycle of length greater than three there is a chord, i.e. an edge connecting two non-consecutive vertices on the cycle.
Chordal graphs are a classic object in graph theory and computer science~\cite{Blair_Peyton__AnIntroductionToChordalGraphsAndCliqueTrees__IVMA_1993}.
In the finite case they are known to be equivalent to the class of graphs representable as a family of subtrees of a tree~\cite{Gavril__TheIntersectionGraphsOfSubtreesInTreesAreExactlyTheChordalGraphs__JCTB_1974}.
A finite and connected chordal graph has natural representations of this form: so-called clique trees, which form a subclass of the spanning trees of its clique graph.
\par
This work investigates clique trees of infinite locally finite chordal graphs.
We show their existence and extend various classic characterisations of clique trees from the finite to the infinite case.
\par
Our core contribution is a local partition of the edge set of the clique graph and a corresponding set of constraints, one for each element of the partition, which a clique tree has to fulfil.
This characterises the clique trees by a bijection with the product of the local choices.
See Section~\ref{sec_characterisation}.
Each constraint only depends on the edges within its partition element, whence the constraints are satisfied or violated independently of each other.
Section~\ref{sec_bijection} presents a purely combinatorial and local construction of a clique tree by fixing a satisfying subset of the edges in each element of the partition.
\par
In the case of a finite chordal graph, our main result gives rise to an enumeration of the clique trees, see Section~\ref{sec_enumeration}.
It is equivalent to a prior enumeration via a local partitioning of constraints by Ho and Lee~\cite{Ho_Lee__CountingCliqueTreesAndComputingPerfectEliminationSchemesInParallel__IPL_1989}.
While their partition is indexed by the minimal vertex separators of the chordal graph, ours is indexed by certain families of cliques.
We recover the minimal vertex separators as intersections of the cliques within those families, thus demonstrating the equivalence of the two approaches.
Section~\ref{sec_min_sep} shows this bijection.
As a corollary, we identify the reduced clique graph with the union of all clique trees, extending a result in~\cite{Galinier_Habib_Paul__ChordalGraphsAndTheirCliqueGraphs__LNCS_1995} to infinite graphs.
\par
Classic characterisations~\cite{Blair_Peyton__AnIntroductionToChordalGraphsAndCliqueTrees__IVMA_1993} of clique trees of finite chordal graphs relate various properties of a clique tree to minimal vertex separators of the original graph, or demand maximality with respect to particular edge weights in the clique graph, or describe properties of paths in the tree, among others.
They contain obstacles to an immediate extension to the infinite case, though.
Either their range is unbounded, or the conditions overlap, or the proof depends on the finite setting or they make no sense at all in an infinite setting (such as maximality with respect to edge weights).
Often these obstacles can be overcome by passing to suitable local conditions, i.e.\ conditions depending only on finitely many vertices or edges.
The partition of the edge set allows us to do that.
Consequently, in Section~\ref{sec_classic_characterisations}, we extend several classic characterisations or sensible versions thereof to the infinite case.

\section{Notation and basics}
\label{sec_basics}

\subsection{Graphs}
\label{sec_graphs}
\par
Throughout this paper we consider locally finite multigraphs, that is, all vertex degrees are finite.
We say \emph{graph}, if we exclude multiple edges.
Let $G$ be a multigraph with vertex set $V$.
For $W\subseteq V$, denote by $G\InducedBy{W}$ the submultigraph of $G$ induced by $W$.
For an equivalence relation $\sim$ on $V$, denote by $G\ContractedBy{\sim}$ the multigraph resulting from contracting each equivalence class of $\sim$ to a single vertex.
It may contain loops and multiple edges, even if $G$ does not.
For $W\subseteq V$, let $G\ContractedBy{W}$ be the multigraph with only $W$ contracted to a single vertex, and, for $W_1,\dotsc,W_k$ disjoint subsets of $V$, let $G\ContractedBy{\Set{W_1,\dotsc,W_k}}$ be the multigraph resulting from $G$ by contracting each $W_i$ to a single vertex.
If we speak of the graph $G\ContractedBy{\sim}$ (or one of the above variants), then we mean the graph underlying the multigraph $G\ContractedBy{\sim}$, including possible loops.
\par
A multigraph is \emph{complete}, if all vertices are adjacent to each other.
We say that $W\subseteq V$ is \emph{complete}, if $G\InducedBy{W}$ is complete.
A \emph{clique} is a maximal complete subset of $V$.
Denote by $\CompleteSetsOfG$ and $\CliquesOfG$ the set of all complete vertex subsets and cliques of $G$ respectively.
The \emph{clique graph} $\CliqueGraphOfG$ of $G$ has vertex set $\CliquesOfG$ and an edge for every pair of cliques with non-empty intersection.
As $G$ is locally finite, all its cliques are finite and every vertex is contained in only a finite number of cliques.
Whence, the clique graph $\CliqueGraphOfG$ is locally finite, too.
\par
A \emph{tree} $T$ is a connected and acyclic graph.
A subgraph of $G$ is \emph{spanning}, if it has the same vertex set as $G$.
Denote by $\SpanningTreesOf{G}$ the \emph{set of spanning trees} of $G$.

\subsection{The lattice of clique families}
\label{sec_cliquefamilies}
\par
For $C\in\CompleteSetsOfG$, the \emph{clique family generated by $C$} is
\begin{equation*}
 \CFGeneratedBy{C}
 := \Set{K \in \CliquesOfG \mid C \subseteq K}\,.
\end{equation*}
Clique families are always non-empty.
Generation is \emph{contravariant}, as
\begin{equation}\label{eq_generation_contravariance}
 C\subseteq C' \Then \CFGeneratedBy{C'}\subseteq\CFGeneratedBy{C}\,.
\end{equation}
The largest clique family is $\CFGeneratedBy{\emptyset} = \CliquesOfG$.
It is infinite if and only if $G$ is infinite, and in this case, it is the only infinite clique family.
For $v\in G$, we abbreviate $\CFGeneratedBy{\Set{v}}$ to $\CFGeneratedBy{v}$.
These are the building blocks of all finite clique families:
\begin{equation}\label{eq_cliqueFamilyViaIntersection}
 \CFGeneratedBy{C}=\bigcap_{v\in C} \CFGeneratedBy{v}\,.
\end{equation}
Let $F$ be a clique family.
Every $C\in\CompleteSetsOfG$ with $\CFGeneratedBy{C}=F$ is a \emph{generator} of $F$.
There is a \emph{maximal generator} of $F$ with respect to set inclusion:
\begin{equation}\label{eq_maximalGenerator}
 \MaxGeneratorOf{F}
 := \bigcap_{K \in F} K
 = \bigcup_{\CFGeneratedBy{C}= F} C\,.
\end{equation}
In particular, we have
\begin{equation}\label{eq_maxGeneratorRegeneratesCliqueFamily}
 \CFGeneratedBy{\MaxGeneratorOf{F}} = F\,.
\end{equation}
It is also immediate that the intersection of two clique families $F_1$ and $F_2$ is again a clique family, more precisely
\begin{equation}
F_1 \cap{}F_2 = \CFGeneratedBy{\MaxGeneratorOf{F_1} \cup \MaxGeneratorOf{F_2}}.
\end{equation}

\par
The sets of generators of two clique families coincide, if and only if the clique families are equal, and are disjoint otherwise.
This follows from the equivalence relation $\sim$ on $\CompleteSetsOfG$ given by $C_1\sim C_2 \Iff \CFGeneratedBy{C_1}=\CFGeneratedBy{C_2}$.
An equivalence class of $\sim$ corresponds to the set of generators of a clique family.

\begin{proposition}\label{prop_cliqueEdge_maxGenerator}
Choose distinct $K_1,K_2\in\CliquesOfG$.
There is an edge $K_1 K_2\in\CliqueGraphOfG$, if and only if $\emptyset\not=K_1\cap K_2=\MaxGeneratorOf{\CFGeneratedBy{K_1\cap K_2}}$.
\end{proposition}

\begin{proof}
We have an edge $K_1 K_2\in\CliqueGraphOfG$, if and only if $K_1\cap K_2\not=\emptyset$.
Thus, $F:=\CFGeneratedBy{K_1\cap K_2}$ is finite and we have
\begin{equation*}
 \emptyset
 \not= K_1\cap K_2
 \subseteq \bigcup_{F(C) = F} C
 \stackrel{\eqref{eq_maximalGenerator}}{=}\MaxGeneratorOf{F}
 \stackrel{\eqref{eq_maximalGenerator}}{=}\bigcap_{K\in F} K
 \subseteq K_1\cap K_2
 \,.
 \qedhere
\end{equation*}
\end{proof}
\par
Let $\CliqueFamiliesOfG$ be the set of clique families of $G$.
The clique families $\CliqueFamiliesOfG$ form a lattice with respect to set inclusion.
Equation~\eqref{eq_cliqueFamilyViaIntersection} implies that all chains in the lattice are finite.
We use this fact to reason inductively over this lattice.

\section{Infinite clique trees}
\label{sec_clique_trees}

\subsection{Chordal graphs and subtree representations}
\label{sec_chordal}
\par
Our main reference for basic facts about chordal graphs is~\cite{Blair_Peyton__AnIntroductionToChordalGraphsAndCliqueTrees__IVMA_1993}.
A \emph{chordal graph} contains no induced cycle of length greater than $3$.
In other words, the induced graph of every cycle of length greater than $3$ contains a \emph{chord}, an edge connecting two non-consecutive vertices of the cycle.
Throughout this work, we assume that chordal graphs are connected, locally finite and do not contain loops.
\par
Let $T$ be a tree and denote by $\mathcal{T}$ the family of subtrees of $T$.
A function $t: V \to \mathcal{T}$ is a \emph{subtree representation of $G$ on $T$}, if $v_{1}v_{2} \in G\Leftrightarrow t(v_{1}) \cap t(v_{2}) \neq \emptyset$.
A finite graph is chordal, if and only if it has a subtree representation on some tree~\cite{Gavril__TheIntersectionGraphsOfSubtreesInTreesAreExactlyTheChordalGraphs__JCTB_1974}.
This remains true for locally finite infinite graphs, but there are examples of countable and non-locally finite chordal graphs which do not admit a subtree representation~\cite{Halin__OnTheRepresentationOfTriangulatedGraphsInTrees__EJC_1984}.
\par
A special kind of subtree representation is given by clique trees.
Clique trees are subtree representations on spanning trees of the clique graph $\CliqueGraphOfG$.
Their existence for finite chordal graphs is a classic result~\cite{Gavril__TheIntersectionGraphsOfSubtreesInTreesAreExactlyTheChordalGraphs__JCTB_1974}.

\begin{definition}\label{def_cliqueTree}
Let $G$ be a chordal graph.
A spanning tree $T$ of $\CliqueGraphOfG$ is called a \emph{clique tree} of $G$, if
\begin{equation}\label{eq_cliqueTree_inducedSubTree}
 \forall\,v\in V:\quad T\InducedBy{\CFGeneratedBy{v}}\text{ is a tree.}
\end{equation}
\end{definition}
A clique tree $T$ represents $G$ via the subtree map $v\mapsto T\InducedBy{\CFGeneratedBy{v}}$, where $\CFGeneratedBy{v}$ is the clique family generated by $\Set{v}$.
Let $\CliqueTreesOfG$ be the set of clique trees of $G$.
\par
The following sections show not only the existence of clique trees of infinite chordal graphs, but a way of constructing them from independent local pieces.
The classic recursive construction in~\cite{Gavril__TheIntersectionGraphsOfSubtreesInTreesAreExactlyTheChordalGraphs__JCTB_1974} depends on the finiteness of the graph to terminate and does not give any indication of how to obtain an independent construction for non-adjacent parts of the chordal graph, a natural goal given the tree-like structure of chordal graphs.

\subsection{Existence of clique trees}
\label{sec_existence}

A first existence result stems from an implicit construction by a limiting procedure.
Explicit local constructions follow in Section~\ref{sec_bijection}.

\begin{theorem}\label{thm_existence}
Every infinite, locally finite chordal graph has a clique tree.
\end{theorem}

\begin{proof}
\par
We use a compactness argument, which is a standard approach in infinite graph theory (c.f.~\cite[Chapter~8.1]{Diestel__GraphTheory__Springer_2005}).
Arguments of this type are often useful to obtain a result for infinite graphs from its finite counterpart.
\par
Let $G$ be the chordal graph.
Let $(v_n)_{n\in\NN}$ be an enumeration of the vertices of $G$ such that $V_n:=\Set{v_1,\dotsc,v_n}$ is connected for each $n\in\NN$.
Let $V_n':=\bigcup_{v\in V_n}\bigcup_{K \in \CFGeneratedBy{v}} K$.
In other words, $V_n'$ contains all vertices that lie in a common clique with some vertex in $V_n$, or equivalently, $V_n$ and its neighbours.
The graph $G_n:=G\InducedBy{V_n'}$ is connected and chordal.
\par
For $v\in V_n$, the clique family $\CFGeneratedBy{v}$ (with respect to $G$) is also a clique family of $G_n$, since $G_n$ contains all vertices in cliques containing $v$.
For $n\in\NN$, let $S_n$ be a clique tree of $G_n$.
Because every clique of $G_n$ is also a clique of $G$, we may interpret $S_n$ as a subtree $T_n$ of $\CliqueGraphOfG$.
\par
Define a subgraph $T$ of $\CliqueGraphOf{G}$ as follows.
By the local finiteness of $G$ and thus $\CliqueGraphOf{G}$, there is an infinite subsequence $(T_{n}^{1})_{n\in\NN}$ of $(T_{n})_{n\in\NN}$ which contain the same edges of $\CliqueGraphOf{G}\InducedBy{\CFGeneratedBy{v_{1}}}$.
Add those edges to $T$.
Next, choose an infinite sub-subsequence $(T_{n}^{2})_{n\in\NN}$ of $(T_{n}^{1})_{n\in\NN}$ such that all elements of the sequence $(T_{n}^{2})_{n\in\NN}$ contain the same edges of $\CliqueGraphOf{G}\InducedBy{\CFGeneratedBy{v_{2}}}$.
Proceed inductively.
\par
By construction, $T\InducedBy{\CFGeneratedBy{v}}$ is a tree, for each $v\in V$.
It remains to verify that $T$ is a tree as well.
The trees corresponding to $v$ and $w$ overlap, if and only if $\CFGeneratedBy{v} \cap \CFGeneratedBy{w} \neq \emptyset$, equivalent to $vw\in G$.
Hence $T$ is connected, because $G$ was assumed to be so.
If $T$ contains a cycle $C$, then it lies in $\CliqueGraphOf{G_m}$, for some $m\in\NN$.
Hence, $C$ is a cycle in the tree $T_1^m\InducedBy{\CliqueGraphOf{G_m}}$, a contradiction.
\end{proof}

\subsection{Local characterisation via clique families}
\label{sec_characterisation}
\par
We show how to construct clique trees of locally finite graphs from small local pieces.
Those pieces live on domains defined in terms of the clique families.
\par
For $F \in \CliqueFamiliesOfG$, let $\StrictSubfamilyGraphOfCF$ be the subgraph of $\CliqueGraphOfG\InducedBy{F}$ with vertex set $F$ and an edge $K_1 K_2\in\StrictSubfamilyGraphOfCF$, if $\CFGeneratedBy{K_1\cap K_2}\subsetneq F$, equivalent to $K_1\cap K_2\supsetneq\MaxGeneratorOf{F}$ by Proposition~\ref{prop_cliqueEdge_maxGenerator}.
Intuitively, the graph $\StrictSubfamilyGraphOfCF$ connects cliques in $F$ whose intersection is ''larger than necessary'', i.e.  their intersection contains a vertex which is not contained in every clique in $F$.
Let $\CFEquivalenceRelation$ be the equivalence relation whose classes are the connected components of $\StrictSubfamilyGraphOfCF$, and let $\CFEquivalenceClassOf{K}$ denote the equivalence class of $K$ with respect to  $\CFEquivalenceRelation$.
This permits to characterise a clique tree in a finer-grained manner than~\eqref{eq_cliqueTree_inducedSubTree}.

\begin{theorem}\label{thm_characterisation}
Let $G$ be a chordal graph.
A spanning subgraph $T$ of $\CliqueGraphOfG$ is a clique tree of $G$, if and only if it satisfies one of the following equivalent conditions:
\begin{subequations}\label{eq_characterisation}
\begin{align}
 \label{eq_characterisation_inducedTree}
 \forall F\in\CliqueFamiliesOfG:
 \quad
 &T\InducedBy{F}\text{ is a tree,}
 \\
 \label{eq_characterisation_contractedTree}
 \forall F\in\CliqueFamiliesOfG:
 \quad
 &T\InducedBy{F}\ContractedBy{\CFEquivalenceRelation}
 \text{ without its loops is a tree.}
\end{align}
\end{subequations}
\end{theorem}
\par
Note that, only~\eqref{eq_characterisation_inducedTree} says directly that $T=T\InducedBy{\CliquesOfG}$ is a tree.
In~\eqref{eq_characterisation_contractedTree}, this fact is not so obvious, but follows from an inductive bottom-up construction.
The advantage of~\eqref{eq_characterisation_contractedTree} is that it allows to compose a clique tree from trees on smaller parts of the clique graph.
In Section~\ref{sec_bijection}, we see that these parts do not overlap.
Thus, we may pick the trees in~\eqref{eq_characterisation_contractedTree} independently of each other.
Consequently, we construct parts of a clique tree locally without knowing the global structure.
\par
Before we give a proof of Theorem~\ref{thm_characterisation} in Section~\ref{sec_proof}, we formulate and prove some auxiliary results in Section~\ref{sec_tree_combination}.

\subsection{Combining trees}
\label{sec_tree_combination}

\begin{lemma}
\label{lem_severalContractions}
Let $G$ be a finite graph with vertex set $V$.
Let $V_{1},V_{2},\dotsc,V_{k}$ be disjoint subsets of $V$.
Every choice of two of the following statements implies the third one:
\begin{subequations}\label{eq_severalContractions}
\begin{gather}
 \label{eq_severalContractions_full}
 G\text{ is a tree,}
 \\
 \label{eq_severalContractions_contracted}
 \forall 1\le i\le k:\quad  G\ContractedBy{\Set{V_1,\dotsc,V_i}}\text{ without its loops is a tree,}
 \\
 \label{eq_severalContractions_restricted}
 \forall 1\le i\le k:\quad G\InducedBy{V_i}\text{ is a tree.}
\end{gather}
\end{subequations}
\end{lemma}

\begin{proof}
A cycle is non-trivial, if it has length greater than $2$.
\par
\underline{\eqref{eq_severalContractions_full} and~\eqref{eq_severalContractions_contracted} $\Then$~\eqref{eq_severalContractions_restricted}}:
$G\InducedBy{V_i}$ does not contain a cycle.
It must be connected, because otherwise there would be a path between two of its connected components in $G$ contracting to a non-trivial cycle in $G\ContractedBy{\Set{V_1,\dotsc,V_i}}$.
\par
\underline{\eqref{eq_severalContractions_full} and~\eqref{eq_severalContractions_restricted} $\Then$~\eqref{eq_severalContractions_contracted}}:
$G\ContractedBy{\Set{V_1,\dotsc,V_i}}$ is connected as a contraction of a connected graph.
Because we only contract connected sets, a non-trivial cycle in $G\ContractedBy{\Set{V_1,\dotsc,V_i}}$ corresponds to a non-trivial cycle in $G$.
\par
\underline{\eqref{eq_severalContractions_contracted} and~\eqref{eq_severalContractions_restricted} $\Then$~\eqref{eq_severalContractions_full}}:
$G$ is connected, because it is obtained from the connected graph $G\ContractedBy{\Set{V_1,\dotsc,V_k}}$ by expanding $i$ vertices into the connected graphs $G\InducedBy{V_i}$.
There cannot be a cycle in $G$, because such a cycle would either lie in some $G\InducedBy{V_i}$ or contract to a non-trivial cycle in $G\ContractedBy{\Set{V_1,\dotsc,V_k}}$.
\end{proof}

\begin{lemma}
\label{lem_treeIntersection}
Let $T$ be a tree with vertex set $V$ and $V_1,V_2\subseteq V$ with non-empty intersection.
If $T\InducedBy{V_1}$ and $T\InducedBy{V_2}$ are trees, then $T\InducedBy{V_1\cap V_2}$ is also a tree.
\end{lemma}

\begin{proof}
Obviously, there is no cycle in $T\InducedBy{V_1\cap V_2}$.
To see that it is connected, observe that for any two vertices $u,v \in V_1\cap V_2$ there are unique $u$-$v$-paths in $T$, $T\InducedBy{V_1}$ and $T\InducedBy{V_2}$.
Those paths coincide and are in $T\InducedBy{V_1\cap V_2}$.
\end{proof}

The following lemma is specific to the situation of clique trees of chordal graphs.
It contains key steps of the proof of Theorem~\ref{thm_characterisation}.

\begin{lemma}
\label{lem_no_cycle_in_ccs}
Let $G$ be chordal graph.
Let $K$ be a clique of $G$ and let $F \in \CliqueFamiliesOfG$ be a clique family containing $K$.
Let $S$ be a subgraph of $\StrictSubfamilyGraphOfCF$ with vertex set $\CFEquivalenceClassOf{K}$.
If, for each clique family $F'\subsetneq F$ with $F' \subseteq \CFEquivalenceClassOf{K}$, $S\InducedBy{F'}$ is a tree, then $S$ is a tree.
\end{lemma}

\begin{proof}

\par
First note that every clique family $F'\subsetneq F$ is either fully contained in $\CFEquivalenceClassOf{K}$ or disjoint from it.
Indeed, any two cliques $K_1,K_2\in{}F'$ satisfy $K_1 \cap K_2 \supseteq C(F') \supsetneq C(F)$.
Thus, they are connected by an edge in $\StrictSubfamilyGraphOfCF$.

\par
\underline{$S$ is connected}:
Let $K'$ be a vertex of $S$.
Because $K\CFEquivalenceRelation K'$, there is a $K'$-$K$-path $K_{0}\dotso K_{n}$ in $\StrictSubfamilyGraphOfCF$.
The definition of edges in $\StrictSubfamilyGraphOfCF$ implies that $K_i\cap K_{i-1}\supsetneq\MaxGeneratorOf{F}$.
Thus, $F_i:=\CFGeneratedBy{K_i\cap K_{i-1}}\subsetneq F$ and $S\InducedBy{F_{i}}$ is a tree containing a $K_{i-1}$-$K_i$-path $P_i$.
The union of $P_1,\dotsc,P_n$ contains a $K$-$K'$-path in $S$.

\par
\underline{$S$ is acyclic}:
Let $\MaxSubfamiliesOfF$ be the set of maximal strict clique subfamilies of $F$.
We say that $F'\in\MaxSubfamiliesOfF$ \emph{covers an edge} $e\in{}S$, if $e\in{}S\InducedBy{F'}$.
If $e$ is an edge of $S$, then it is also an edge of $\StrictSubfamilyGraphOfCF$.
In particular, its endpoints correspond to cliques whose intersection generates a strict subfamily $F' \subsetneq F$.
Hence, each edge in $S$ is covered by some clique family in $\MaxSubfamiliesOfF$.
A subset of $\MaxSubfamiliesOfF$ \emph{covers a subgraph} of $\StrictSubfamilyGraphOfCF$, if each edge of the subgraph is covered by at least one element of the subset.

\par
Assume for a contradiction that $S$ contains a non-trivial cycle.
Let $R\le\Cardinality{\MaxSubfamiliesOfF}$ be the minimal cardinality of a subset of $\MaxSubfamiliesOfF$ covering a non-trivial cycle of $S$.
For a non-trivial cycle $Z$ with cover $\Set{F_1,\dotsc,F_R}$, we say that an edge of $Z$ is \emph{uniquely covered} by $F_i$, if it is covered by $F_i$ and not covered by any other $F_i$, for $j \neq i$.
Denote by $U_i$ the set of edges uniquely covered by $F_i$.
The minimality of $R$ implies that all $U_i$ are non-empty.
Call $F_i$ \emph{nice} for $Z$, if all edges of $U_i$ are contained in the same connected component of $Z\InducedBy{F_i}$.
We claim that there exists a non-trivial cycle $Z$ with cover $\Set{F_1,\dotsc,F_R}$ such that
\begin{equation}
\label{eq_cover_nice}
 \forall{}1\le{}i\le{}R:\quad{}F_i\text{ is nice for }Z
 \,.
\end{equation}

\par
The statement~\eqref{eq_cover_nice} is trivial for $R=1$.
For $R\ge{}2$, we transform $Z$ into a non-trivial cycle with the same cover $\Set{F_1,\dotsc,F_R}$ and fulfilling~\eqref{eq_cover_nice} in two steps:
\begin{enumerate}
\item\label{step_f1_nice}
There is a cycle $Z'$ covered by $\Set{F_1,\dotsc,F_R}$ with $F_1$ being nice for $Z'$.
\item\label{step_all_nice}
If $F_1$ is nice, then $F_2,\dotsc,F_R$ are nice, too.
\end{enumerate}

\par
\underline{Step~\ref{step_f1_nice}}:
If $Z\InducedBy{F_1}$ is connected, then let $Z' := Z$.
Otherwise, pick edges $e,f$ in different connected components of $Z\InducedBy{F_1}$ such that there is a non-trivial path $P$ in $Z$ without an edge covered by $F_1$ and connecting an endpoint of $e$ and $f$ each.
The path $P$ is covered by $\Set{F_2,\dotsc,F_R}$.
Because the endpoints of $P$ are distinct and contained in $F_1$, and $F_1$ is connected, there is a non-trivial path $Q$ with the same endpoints as $P$ and covered by $F_1$.
Clearly, $P$ and $Q$ are edge-disjoint and $Z' := P \cup Q$ is a non-trivial cycle covered by $\Set{F_1,\dotsc,F_R}$.
As $Q$ is connected, $F_1$ is nice for $Z'$.

\par
\underline{Step~\ref{step_all_nice}}:
Assume that $F_1$ is nice but some other $F_i$ is not.
Let $e,f \in U_i$ be edges contained in different connected components of $Z\InducedBy{F_i}$.
The graph $Z-\Set{e,f}$ consists of two non-trivial paths $P_1$ and $P_2$.
Without loss of generality, all of $U_1$ lies in $P_1$ and $P_2$ is covered by $\Set{F_2,\dotsc,F_R}$.
There is a non-trivial path $Q$ in $S\InducedBy{F_i}$ connecting the endpoints of $P_1$.
Because $P_2$ contains an edge not covered by $F_i$ and $Q$ is covered by $F_i$, we have $P_2\neq Q$.
Hence, $P_2\cup Q$ contains a non-trivial cycle covered by $\Set{F_2,\dotsc,F_R}$.
This contradicts the minimality of $R$.

\par
Let $\ModuloRDistance{i-j}$ be the modulo $R$ distance between $i$ and $j$.
We claim that we may reorder the clique families in a cover of a cycle fulfilling~\eqref{eq_cover_nice} such that
\begin{equation}
\label{eq_cover_overlap_modulo}
 \forall{}1\le{}i<j\le{}R:
 \quad
 F_i\cap{}F_j\not=\emptyset
 \Iff{}
 \ModuloRDistance{i-j}\le{}1
\end{equation}

\par
For each $i$, choose an edge $e_i \in U_i$.
Order the sets $F_i$ according to the cyclic order of the edges $e_i$.
The $F_i$'s niceness by~\eqref{eq_cover_nice} guarantees that the order is independent of the choice of edges.
Without loss of generality assume that the order is $1,2,\dotsc,R$.
\par
First, we show that distant clique families are disjoint.
Consider $i,j$ with $\ModuloRDistance{i-j}>1$, whence $R\ge{}4$ holds.
Assume that $F_i\cap{}F_j\not=\emptyset$.
The graph $Z-\Set{e_i,e_j}$ consists of two disjoint non-trivial paths $P_1$ and $P_2$.
Because $F_i\cap{}F_j\not=\emptyset$, there is a non-trivial path $Q$ in $S\InducedBy{F_i\cup{}F_j}$ connecting $e_i$ and $e_j$.
Niceness implies that there are distinct clique families $F_{k_1}$ and $F_{k_2}$ with $k_1,k_2\not\in\Set{i,j}$ such that $U_{k_1}\subseteq{}P_1$ and $U_{k_2}\subseteq{}P_2$.
Hence, $P_1\not=Q$ and $P_2\not=Q$.
Therefore, the union of $P_1$, $Q$ and $\Set{e_i,e_j}$ contains a non-trivial cycle covered by $\Set{F_1,\dotsc,F_R}\setminus\Set{F_{k_2}}$.
This contradicts the minimality of $R$.
\par
Second, we show that close clique families overlap.
For $i$, let $F_j$ and $F_k$ be the two clique families with $\ModuloRDistance{i-j} = 1 = \ModuloRDistance{i-k}$.
Without loss of generality, assume that $F_i\cap{}F_j=\emptyset$.
This can only happen if $R\ge{}3$.
The connected component of $Z\InducedBy{U_i}$ containing $e_i$ is a path $P$.
The endpoints of $P$ are contained in $F_k$, because of disjointness they can not be part of $F_l$, with $l\not\in\Set{i,j,k}$, nor in $F_j$, by assumption.
Connect these endpoints by a path $Q$ in $S\InducedBy{F_k}$.
The paths $P$ and $Q$ are distinct.
Their union contains a non-trivial cycle covered by $\Set{F_i,F_k}$.
This contradicts the assumption that $R\ge{}3$.

\par
This completes the proof of the claim~\eqref{eq_cover_overlap_modulo}.
From here on, we work with a cycle $Z$ satisfying both~\eqref{eq_cover_nice} and~\eqref{eq_cover_overlap_modulo}.
We proceed by case analysis on $R$.

\par
\underline{Case $R=1$}:
The tree $S\InducedBy{F_1}$ cannot contain the cycle $Z$.

\par
\underline{Case $R=2$}:
The fact that $F_1$ and $F_2$ are nice by~\eqref{eq_cover_nice} implies that there exist distinct $K_1,K_2\in{}Z\InducedBy{F_1\cap{}F_2}$ splitting $Z$ into two non-trivial paths $P_1$ and $P_2$, disjoint except in $\Set{K_1,K_2}$, such that $U_1\subseteq{}P_1$ and $U_2\subseteq{}P_2$.
As $S\InducedBy{F_1\cap{}F_2}$ is a tree, there exists a unique path $Q$ between $K_1$ and $K_2$ in $S\InducedBy{F_1\cap{}F_2}$.
As $U_1\subseteq{}P_1$, $P_1\not=Q$.
Hence, the union of $P_1$ and $Q$ contains a non-trivial cycle covered by $F_1$.
This is a contradiction to $S\InducedBy{F_1}$ being a tree.

\par
\underline{Case $R=3$}:
The fact that $F_1,F_2$ and $F_3$ are nice by~\eqref{eq_cover_nice} implies that there exist distinct $K_{12}\in{}Z\InducedBy{F_1\cap{}F_2}$, $K_{23}\in{}Z\InducedBy{F_2\cap{}F_3}$ and $K_{13}\in{}Z\InducedBy{F_1\cap{}F_3}$ splitting $Z$ into three non-trivial paths $P_1$, $P_2$ and $P_3$, disjoint except in $\Set{K_{12},K_{23},K_{13}}$, such that, for all $1\le{}i\le{}3$, $U_i\subseteq{}P_i$.
\par
Let $C_{i}:=\MaxGeneratorOf{F_i}$ be the maximal generator of the clique family $F_i$.
By~\eqref{eq_cover_overlap_modulo}, there is a clique in $F_i \cap{}F_j$.
Hence, $C_{i}\cup{}C_{j}$ is complete for all $1\le i,j\le 3$.
Thus, $C:=\bigcup_{i=1}^3 C_{i}$ is complete and $\emptyset\not=\CFGeneratedBy{C}\subseteq\bigcap_{i=1}^3 F_i$.
\par
Fix $K'\in{}\bigcap_{i=1}^3 F_i$.
As $S\InducedBy{F_1\cap{}F_2}$ is a tree, there exists a unique path $P_{12}$ between $K_{12}$ and $K'$ in $S\InducedBy{F_1\cap{}F_2}$.
Likewise, there is a unique path $P_{23}$ in $S\InducedBy{F_2\cap{}F_3}$ between $K_{23}$ and $K'$ and a unique path $P_{13}$ in $S\InducedBy{F_1\cap{}F_3}$ between $K_{13}$ and $K'$.
The union of $P_1$, $P_{12}$ and $P_{13}$ is covered by $F_1$ and contains a non-trivial cycle because $U_1\subseteq{}P_1$.
This is a contradiction to $S\InducedBy{F_1}$ being a tree.

\par
\underline{Case $R\ge{}4$}:
Again, let $C_{i}:=\MaxGeneratorOf{F_i}$ and let $D_i := C_i \setminus \MaxGeneratorOf{F}\not=\emptyset$.
For every vertex $v \in D_i$, the set $\MaxGeneratorOf{F} \uplus \Set{v}$ generates a clique family satisfying $F_i \subseteq \CFGeneratedBy{\MaxGeneratorOf{F} \uplus \Set{v}} \subsetneq F$.
Because $F_i$ is a maximal strict subfamily of $F$, we infer that $F_i = \CFGeneratedBy{\MaxGeneratorOf{F} \uplus \Set{v}}$.
In particular the sets $D_i$ are disjoint, because $v \in D_i \cap D_j$ implies that $F_i = \CFGeneratedBy{\MaxGeneratorOf{F} \uplus \Set{v}} = F_j$.
\par
We investigate the edges between the sets $D_i$.
If $\ModuloRDistance{i-j}=1$, then~\eqref{eq_cover_overlap_modulo} implies the existence of $K\in{}F_i \cap{}F_j$.
Hence, $D_i \cup D_j\subseteq{}C_{i}\cup{}C_j\subseteq{}K$ and $G\InducedBy{D_i \cup D_j}$ is complete.
If $\ModuloRDistance{i-j}>1$, then assume that there is an edge $v_i v_j$ with $v_i \in D_i$ and $v_j \in D_j$.
Then, $\MaxGeneratorOf{F} \uplus \Set{v_i,v_j}$ is complete and $F_i \cap{}F_j$ contains a clique, a contradiction to~\eqref{eq_cover_overlap_modulo}.
Thus, $G\InducedBy{D_i \cup D_j}$ contains no edges between $D_i$ and $D_j$, if $\ModuloRDistance{i-j}>{}1$.
\par
For $1 \leq i \leq R$, choose $v_{i}\in{}D_i$.
The induced subgraph $G\InducedBy{\Set{v_1,\dotsc,v_R}}$ is a chordless cycle with length $R\ge{}4$ and contradicts the chordality of $G$.
\end{proof}

\subsection{Proof of Theorem~\ref{thm_characterisation}}
\label{sec_proof}
\par
We prove the equivalences~\eqref{eq_cliqueTree_inducedSubTree}~$\Iff$~\eqref{eq_characterisation_inducedTree} and~\eqref{eq_characterisation_inducedTree}~$\Iff$~\eqref{eq_characterisation_contractedTree}.
For convenience, we restate them.
A spanning subgraph $T$ of $\CliqueGraphOfG$ is a clique tree of $G$, if and only if it satisfies the following equivalent conditions:
\begin{gather}
 T\text{ is a tree and }
 \forall v\in{}V:
 \quad
 T\InducedBy{\CFGeneratedBy{v}}\text{ is a tree,}
 \tag{\ref{eq_cliqueTree_inducedSubTree}}
 \\
 \forall F\in\CliqueFamiliesOfG:
 \quad
 T\InducedBy{F}\text{ is a tree,}
 \tag{\ref{eq_characterisation_inducedTree}}
 \\
 \forall F\in\CliqueFamiliesOfG:
 \quad
 T\InducedBy{F}\ContractedBy{\CFEquivalenceRelation}
 \text{ without loops is a tree.}
 \tag{\ref{eq_characterisation_contractedTree}}
\end{gather}
\par
\underline{\eqref{eq_cliqueTree_inducedSubTree} $\Then$~\eqref{eq_characterisation_inducedTree}}:
If $F = \CliquesOfG$ or $F = \CFGeneratedBy{v}$, for some vertex $v \in V$, then $T\InducedBy{F}$ is a tree.
Assume that~\eqref{eq_characterisation_inducedTree} does not hold.
The finiteness of chains in $\CliqueFamiliesOfG$ lets us choose a maximal $F\in\CliqueFamiliesOfG$ such that $T\InducedBy{F}$ is not a tree.
Furthermore, each generator of $F$ contains at least two vertices.
Let $C\subseteq\MaxGeneratorOf{F}$ be a minimal generator of $F$.
For every $\emptyset\not=C'\subsetneq{}C$, the contravariance of clique family generation~\eqref{eq_generation_contravariance} implies that $\CFGeneratedBy{C'}$ and $\CFGeneratedBy{C\setminus C'}$ are strictly larger than $F$ and $F = \CFGeneratedBy{C'} \cap \CFGeneratedBy{C\setminus C'}$.
Maximality of $F$ implies that $T\InducedBy{\CFGeneratedBy{C'}}$ and $T\InducedBy{\CFGeneratedBy{C\setminus C'}}$ are trees.
Lemma~\ref{lem_treeIntersection} implies that $T\InducedBy{F}$ is a tree, too.
\par
\underline{\eqref{eq_characterisation_inducedTree} $\Then$~\eqref{eq_cliqueTree_inducedSubTree}}:
Equation~\eqref{eq_characterisation_inducedTree} implies that $T\InducedBy{\CFGeneratedBy{v}}$ is a tree, for each $v\in G$, and that $T\InducedBy{\CliquesOfG}=T$ is a tree.
\par
\underline{\eqref{eq_characterisation_inducedTree} $\Then$~\eqref{eq_characterisation_contractedTree}}:
Let $F \in \CliqueFamiliesOfG$.
Lemma~\ref{lem_no_cycle_in_ccs} together with the assumption that $T\InducedBy{F'}$ is a tree for every $F'\subsetneq F$ implies that $T\InducedBy{\CFEquivalenceClassOf{K}}$ is a tree, for every equivalence class with respect to the relation $\CFEquivalenceRelation$.
If $F \neq \CliquesOfG$, then there are only finitely many equivalence classes.
Hence, we apply Lemma~\ref{lem_severalContractions} to show that $T\InducedBy{F}\ContractedBy{\CFEquivalenceRelation}$ is a tree.
For $F = \CliquesOfG$, we know that $\CFEquivalenceRelation$ is the trivial relation, i.e.\ any two cliques are related.
Whence, $T\InducedBy{F}\ContractedBy{\CFEquivalenceRelation}$ is a single vertex tree.
\par
\underline{\eqref{eq_characterisation_contractedTree} $\Then$~\eqref{eq_characterisation_inducedTree}}:
Assume that there is some $F \in \CliqueFamiliesOfG$ such that $T\InducedBy{F}$ is not a tree.
Choose $F$ minimal with this property.
This is possible because chains in $\CliqueFamiliesOfG$ are finite.
Lemma~\ref{lem_no_cycle_in_ccs} implies that $T\InducedBy{\CFEquivalenceClassOf{K}}$ is a tree for every equivalence class with respect to $\CFEquivalenceRelation$.
Because there are only finitely many equivalence classes and $T\InducedBy{F}\ContractedBy{\CFEquivalenceRelation}$ is a tree,
Lemma~\ref{lem_severalContractions} shows that $T\InducedBy{F}$ is a tree.

\subsection{Edge bijections}
\label{sec_bijection}

\par
In this section we show that the restrictions~\eqref{eq_characterisation_contractedTree} imposed by a clique family and its strict subfamilies are independent of each other.
This allows us to write the set of clique trees as the product of sets of smaller trees, see Theorem~\ref{thm_bijection}.
The product is indexed by the clique families.
For a given clique family, the associated set of trees is independent of the sets of trees for subfamilies of the clique family.
\par
Let $F\in\CliqueFamiliesOfG$.
Recall that $\StrictSubfamilyGraphOfCF$ was defined as the subgraph of $\CliqueGraphOfG\InducedBy{F}$ containing all edges between cliques whose intersection is strictly larger than $\MaxGeneratorOf{F}$.
Let $\SameFamilyGraphOfCF$ be the subgraph of $\CliqueGraphOfG\InducedBy{F}$ containing the remaining edges.
That is, $\SameFamilyGraphOfCF$ contains an edge $K_1 K_2$, if $\CFGeneratedBy{K_1\cap K_2}= F$, or equivalently $K_1\cap K_2=\MaxGeneratorOf{F}$, by Proposition~\ref{prop_cliqueEdge_maxGenerator}.
Intuitively, the graph $\SameFamilyGraphOfCF$ connects cliques in $F$ whose intersection is ''as small as possible'' within $\CliqueGraphOfG\InducedBy{F}$.
It is obvious from the definitions that $\StrictSubfamilyGraphOfCF$ and $\SameFamilyGraphOfCF$ partition the edges of $\CliqueGraphOfG\InducedBy{F}$ into two disjoint sets.
\par
Consider the multigraph $\ContractionGraphOfCF:=\SameFamilyGraphOfCF\ContractedBy{\CFEquivalenceRelation}$, i.e.  all components of $\StrictSubfamilyGraphOfCF$ are contracted to single points.
This graph may contain (multiple) loops.
We use the natural bijection between edges of $\SameFamilyGraphOfCF$ and edges of $\ContractionGraphOfCF$ to label the edges of $\ContractionGraphOfCF$ and differentiate between them.
\par
It is worth noting that $\CliqueGraphOfG\InducedBy{F}\ContractedBy{\CFEquivalenceRelation}$ can be obtained from $\ContractionGraphOfCF$ by adding additional loops.
As a consequence, spanning trees of the two graphs are in one-to-one correspondence.

\begin{proposition}\label{prop_edge_partition}
There is a bijection between the edges of $\CliqueGraphOfG$ and the disjoint union over all clique families $F$ of edges of $\SameFamilyGraphOfCF$.
Via edge-labelling, this extends to the disjoint union of edges of $\ContractionGraphOfCF$.
\begin{equation}\label{eq_edge_bijection}
 \CliqueGraphOfG
 \stackrel{\text{edges}}{=}
   \biguplus_{F\in\CliqueFamiliesOfG}
   \SameFamilyGraphOfCF
 \stackrel{\text{edge-labelling}}{=}
   \biguplus_{F\in\CliqueFamiliesOfG}
   \ContractionGraphOfCF\,.
\end{equation}
\end{proposition}

\begin{proof}
For $K_1 K_2\in\CliqueGraphOfG$, consider the clique family $F:=\CFGeneratedBy{K_1\cap K_2}$.
By Proposition~\ref{prop_cliqueEdge_maxGenerator}, we have $K_1\cap K_2=\MaxGeneratorOf{F}$.
The definition of $\SameFamilyGraphOfCF$ allows $K_1 K_2$ only as an edge in $\SameFamilyGraphOfCF$, but not in any other $\SameFamilyGraphOf{F'}$ with $F' \neq F$.
\end{proof}

\begin{theorem}\label{thm_bijection}
There is a bijection between the clique trees $\CliqueTreesOfG$ and a $\CliqueFamiliesOfG$-indexed product of sets of spanning trees.
For each clique tree, its edges and the edges of the spanning trees in its corresponding $\CliqueFamiliesOfG$-indexed collection are in bijection, too.
\begin{equation}\label{eq_bijection}
 \CliqueTreesOfG
 \stackrel{\text{edge-labelling}}{=}
 \prod_{F\in\CliqueFamiliesOfG}
 \SpanningTreesOf{\ContractionGraphOfCF}\,.
\end{equation}
\end{theorem}
\par
A similar bijection to~\eqref{eq_bijection} between the clique trees of a finite chordal graph and a product of trees indexed by the minimal vertex separators of the graph is already known~\cite{Ho_Lee__CountingCliqueTreesAndComputingPerfectEliminationSchemesInParallel__IPL_1989}.
We discuss their relationship in Section~\ref{sec_min_sep}.

\begin{proof}
\par
Using the bijection from Proposition~\ref{prop_edge_partition}, we split the edges of a clique tree $T\in\CliqueTreesOfG$ into disjoint sets $E_F:=\Set{K_1K_2: K_1K_2\in T,K_1K_2\in\SameFamilyGraphOfCF}$, indexed by $\CliqueFamiliesOfG$.
For $F\in\CliqueFamiliesOfG$, statement~\eqref{eq_characterisation_contractedTree} tells us that $E_F$ labels the edges of a spanning tree of $\ContractionGraphOfCF$.
\par
Conversely, select a spanning tree $T_F\in\SpanningTreesOf{\ContractionGraphOfCF}$, for each $F\in\CliqueFamiliesOfG$.
Let $E$ be the union of their edge-labels.
By Proposition~\ref{prop_edge_partition}, each edge in $E$ appears exactly once as an edge-label of some $T_F$.
By~\eqref{eq_characterisation_contractedTree}, the graph $T:=(\CliquesOfG,E)$ is a clique tree.
\end{proof}

\subsection{Enumerating the clique trees}
\label{sec_enumeration}
\par
In this section, we enumerate the clique trees of a given chordal graph.
We start with a structure statement about the auxiliary multigraphs.

\begin{proposition}\label{prop_completenes}
The multigraph $\ContractionGraphOfCF$ is complete.
\end{proposition}

\begin{proof}
\par
\underline{Case $\MaxGeneratorOf{F}=\emptyset$}:
This only happens, if $G$ contains disjoint cliques and $F=\CliquesOfG$.
In this case, we have $\StrictSubfamilyGraphOf{\CliquesOfG}=\CliqueGraphOfG$ and $\SameFamilyGraphOf{\CliquesOfG}=(\CliquesOfG,\emptyset)$.
As $\CliquesOfG$ forms one equivalence class under $\CFEquivalenceRelation$, $\ContractionGraphOf{\CliquesOfG}$ is a graph with one vertex and no edges.
\par
\underline{Case $\MaxGeneratorOf{F}\not=\emptyset$}:
This implies that $F$ is finite.
For all distinct $K_1,K_2\in F$,
\begin{equation*}
 \emptyset\not=\MaxGeneratorOf{F}
 =\bigcap_{K\in F} K \subseteq K_1\cap K_2\,.
\end{equation*}
Therefore, $\CliqueGraphOfG\InducedBy{F}$ is complete and so is $\ContractionGraphOfCF$.
\end{proof}

An immediate consequence of~\eqref{eq_bijection} is a count of clique trees of a finite chordal graph.
\begin{equation}
 \Cardinality{\CliqueTreesOfG}
 =\prod_{F\in\CliqueFamiliesOfG}
   \Cardinality{\SpanningTreesOf{\ContractionGraphOfCF}}\,.
\end{equation}
The value of $\Cardinality{\SpanningTreesOf{\ContractionGraphOfCF}}$ is explicitly given in terms of the structure of $\ContractionGraphOfCF$ as a complete multigraph via a \emph{matrix-tree theorem} from~\cite{Ho_Lee__CountingCliqueTreesAndComputingPerfectEliminationSchemesInParallel__IPL_1989}.

\begin{corollary}\label{cor_cliqueTreeEnumeration}
Fix $D\in\NN$.
For every finite chordal graph $G$ with maximal degree $D$ and vertices $V$, one can generate $\CliqueTreesOfG$ sequentially with only $O(\Cardinality{V})$ working memory.
\end{corollary}

\begin{proof}
As the degree is uniformly bounded, so are the sizes of a clique (by $D+1$, with equality if all edges incident to a vertex belong to the same clique), a non-trivial clique family $F$ (by $D$, with equality if all edges incident to a vertex belong to different cliques) and its spanning trees $\SpanningTreesOf{\ContractionGraphOfCF}$ (by $D^{D-2}$ via Cayley's formula).
Furthermore, each vertex is only contained in at most $D$ cliques and hence in at most $D$ clique families, so the size of $\CliqueFamiliesOfG$ is linear in $\Cardinality{V}$.
Generate $\SpanningTreesOf{\ContractionGraphOfCF}$, for all $F\in\CliqueFamiliesOfG$.
This takes memory linear in $\Cardinality{V}$, with worst case multiplicative constants given by the bounds above which depend only on $D$.
Iterate in lexicographic order through all the local choices of spanning trees and use~\eqref{eq_bijection} to obtain a clique tree from a full set of local choices.
\end{proof}

For infinite chordal graphs, there is a dichotomy in the number of clique trees.

\begin{corollary}\label{cor_cliqueTreeCountInfinite}
Let $G$ be an infinite chordal graph.
It has either finitely or $2^{\aleph_0}$ many clique trees.
\end{corollary}

\begin{proof}
We look at $\Set{\Cardinality{\SpanningTreesOf{\ContractionGraphOfCF}}}_{F\in\CliqueFamiliesOfG}$.
It is countable, because $\CliqueFamiliesOfG$ is so.
If only a finite number of these numbers are greater than $1$, then the number of clique trees is finite.
If an unbounded number of these numbers are greater than $1$, then there is a countable number of independent choices between more than two spanning trees and the number of clique trees is $2^{\aleph_0}$.
\end{proof}

\subsection{Minimal vertex separators and the reduced clique graph}
\label{sec_min_sep}
\par
As mentioned previously, a bijection indexed by minimal vertex separators and similar to~\eqref{eq_bijection} was given by Ho and Lee~\cite{Ho_Lee__CountingCliqueTreesAndComputingPerfectEliminationSchemesInParallel__IPL_1989}.
Lemma~\ref{lem_min_sep_cf} shows that the minimal vertex separators correspond to the maximal generators of clique families with a non-trivial contribution to the bijection.
As a consequence, the two decompositions coincide.
\par
Following~\cite[Section 2.2]{Blair_Peyton__AnIntroductionToChordalGraphsAndCliqueTrees__IVMA_1993}, we  call $\emptyset\not=W\subseteq V$ a \emph{$v$-$w$-separator}, if $v$ and $w$ lie in different connected components of $G\InducedBy{V\setminus W}$.
We call $\emptyset\not=W\subseteq V$ a \emph{minimal vertex separator}, if there exist vertices $v$ and $w$, such that $W$ is a $v$-$w$-separator and no proper subset of $W$ is a $v$-$w$-separator.
Minimal vertex separators characterise chordal graphs.
The proof of the following result for finite graphs from~\cite{Blair_Peyton__AnIntroductionToChordalGraphsAndCliqueTrees__IVMA_1993} generalises verbatim to infinite graphs.

\begin{theorem}[{\cite[Theorem 2.1]{Blair_Peyton__AnIntroductionToChordalGraphsAndCliqueTrees__IVMA_1993} after~\cite{Dirac__OnRigidCircuitGraphs__Hamburg_1961}}]
\label{thm_min_sep}
A graph is chordal, if and only if every minimal vertex separator is complete.
\end{theorem}
\par
The remainder of this section shows that the minimal vertex separators form a subset of the maximal generators of the clique families.

\begin{lemma}
\label{lem_min_sep_adjacent}
A minimal vertex separator in a chordal graph separates two vertices adjacent to all of it.
\end{lemma}

\begin{proof}
\par
Let $C$ be a minimal $v_1$-$v_2$-separator.
For every $w\in C$, there exists a $w$-$v_1$-path $P_w$ with $P_w\cap C=\Set{w}$.
The path $P_w$ may be assumed to be chordless, i.e.  non-successive vertices are not adjacent.
For each $w\in C$, let $v_1^w$ be the neighbour of $w$ on $P_w$.
Let $V_1:=\Set{v_1^w\mid{}w\in C}\not=\emptyset$.
If we show that one $u_1\in V_1$ fulfils $C\cup\Set{v}\in\CompleteSetsOfG$, then a symmetric argument for a likewise $u_2$ on the $v_2$-side shows that $C$ is a minimal $u_1$-$u_2$-separator.
\par
For each $v\in V_1$, let $C_v:=\Set{w\in C\mid{}vw\in G}$.
In particular, $w\in C_{v_1^w}\neq\emptyset$.
Order $V_1$ by the partial order induced by the subset relation on $\Set{C_v\mid{}v\in V_1}$.
If there exists a unique maximal element $v$ in $V_1$, then, for all $w\in C$, $w\in C_{v_1^w}\subseteq C_v$.
Whence, $C_v=C$ and $C_v\cup\Set{v}\in\CompleteSetsOfG$.
\par
If there exist more than one maximal element in $V_1$, then let $u$ and $v$ be two of them.
This implies that there exist $w_u\in C_u\setminus C_v$ and $w_v\in C_v\setminus C_u$.
Because $w_u$ and $w_v$ lie in $C$, they are connected.
The union of $P_{w_u}$, the $w_u w_v$ edge and $P_{w_v}$ contains an cycle going through $u,w_u,w_v$ and $v$.
As $u$ and $w_v$ are not connected, there must be a chord incident to $w_u$.
Because $P_{w_u}$ is chordless, the other end of the chord must be a vertex  in $P_{w_v}\setminus\Set{w_v,v}$.
Let $z$ be the neighbour of $w_u$ in $P_{w_v}\setminus\Set{w_v,v}$ which lies closest to $v$ (measured along $P_{w_v}$).
Consider the smaller cycle formed by the edges $z w_u$, $w_u w_v$ and $P_{w_v}$ between $w_v$ and $z$.
It contains $z$, $w_u$, $w_v$ and $v$ and has length at least $4$.
But the vertex $w_u$ cannot be incident to a chord, because of the minimality of $z$ and all other vertices lie on the chordless path $P_{w_v}$.
Thus, there cannot be a chord and there cannot be multiple maximal elements of $V_1$.
\end{proof}

\begin{lemma}
\label{lem_min_sep_cf}
A complete set of vertices $C\in\CompleteSetsOfG$ is a minimal vertex separator of $G$, if and only if it is the maximal generator of $\CFGeneratedBy{C}$ and $\ContractionGraphOf{\CFGeneratedBy{C}}$ contains more than one vertex.
\end{lemma}

\begin{proof}
\par
Let $C$ be a minimal vertex separator.
By Theorem~\ref{thm_min_sep}, $C$ is complete.
By Lemma~\ref{lem_min_sep_adjacent}, $C$ separates $v_1$ and $v_2$ such that $C\uplus\Set{v_1}$ and $C\uplus\Set{v_2}$ are complete.
Hence, there are cliques $K_1, K_2\in\CFGeneratedBy{C}=:F$ with $C\uplus\Set{v_1}\subseteq K_1$ and $C\uplus\Set{v_2}\subseteq K_2$.
\par
It is immediate that $C$ is the maximal generator of $F$, because any generator of $F$ is contained in both $K_1$ and $K_2$.
Thus, a bigger generator would give a common neighbour of $v_1$ and $v_2$ outside of $C$, contradicting the fact that $C$ is a $v_1$-$v_2$-separator.
\par
In order to prove that $\ContractionGraphOfCF$ has at least two vertices, it suffices to show that there is no $K_1$-$K_2$-path in $\StrictSubfamilyGraphOfCF$.
So assume that there was such a path $P$.
For each edge $KK'\in P$, there is a vertex $v_{KK'}\in (K\cap K')\setminus C\not=\emptyset$.
The graph $G\InducedBy{\Set{v_1,v_2}\cup\Set{v_{KK'}\mid{}KK'\in P}}$ contains a $v_1$-$v_2$-path.
This contradicts the vertex separator property of $C$.
\par
For the converse implication, let $F$ be a clique family with $\ContractionGraphOfCF$ having more than two vertices.
Choose two distinct vertices $\CFEquivalenceClassOf{K_1}$ and $\CFEquivalenceClassOf{K_2}$ of $\ContractionGraphOfCF$.
It follows that $K_1\not\CFEquivalenceRelation K_2$, implying $K_1\cap K_2=\MaxGeneratorOf{F}=:C$.
Choose $v_1\in K_1\setminus C$ and $v_2\in K_2\setminus C$.
We claim that $C$ is a minimal $v_1$-$v_2$-separator.
\par
Minimality is obvious, as, for every $v\in C$, $v_1 v v_2$ is a path in $G\InducedBy{V\setminus(C\setminus\Set{v})}$.
It remains to show that $C$ separates $v_1$ and $v_2$.
Assume for a contradiction that there is a $v_1$-$v_2$-path $P$ in $G\InducedBy{V\setminus C}$.
For every $w \in P$, there is a minimal $v_1$-$v_2$-separator containing $C \uplus \Set{w}$.
Because minimal vertex separators are complete, $w$ is connected to all of $C$ and $C \uplus \Set{w}\subseteq{}K_w\in\CFGeneratedBy{C}$.
The sequence $(K_w)_{w\in P}$ contains a $K_1$-$K_2$-path in $\StrictSubfamilyGraphOfCF$, contradicting the original choice of $K_1$ and $K_2$ from different connected components.
Therefore, $C$ is a $v_1$-$v_2$-separator.
\end{proof}

The \emph{reduced clique graph}~\cite{Galinier_Habib_Paul__ChordalGraphsAndTheirCliqueGraphs__LNCS_1995} $\ReducedCliqueGraphOfG$ of $G$ is the subgraph of $\CliqueGraphOfG$ retaining only those edges $K_1 K_2$ with $K_1\cap K_2$ a minimal vertex separator.

\begin{theorem}[{Generalisation of~\cite[Theorem 7]{Galinier_Habib_Paul__ChordalGraphsAndTheirCliqueGraphs__LNCS_1995}}]
\label{thm_reducedCliqueGraph}
The set $\Set{K_1\cap K_2\mid K_1 K_2\in T}$ is an invariant of a clique tree $T\in\CliqueTreesOfG$ and equals the set of minimal vertex separators of $G$.
The union of the clique trees of a chordal graph $G$ is the reduced clique graph $\ReducedCliqueGraphOfG$.
\end{theorem}

\begin{proof}
The statements are direct consequences of Lemma~\ref{lem_min_sep_cf} together with the bijection in Theorem~\ref{thm_bijection}.
\end{proof}

\section{Classic characterisations of clique trees}
\label{sec_classic_characterisations}
\par
For finite chordal graphs, there exist other characterisations of clique trees besides~\eqref{eq_cliqueTree_inducedSubTree}.
This section generalises or adapts these results to the infinite case.
The characterisations are the clique intersection property in Theorem~\ref{thm_cip}, the running intersection property in Theorem~\ref{thm_rip} and the maximal weight spanning tree property in Theorem~\ref{thm_max_local}.
\par
A tree $T\in\SpanningTreesOf{\CliqueGraphOfG}$ has the \emph{clique intersection property}, if $K_1\cap K_2\subseteq K_3$ holds, for every three cliques $K_1,K_2,K_3$ with $K_3$ lying on the $K_1$-$K_2$-path in $T$.

\begin{theorem}[Generalisation of the finite case in~{\cite[Section 3.1]{Blair_Peyton__AnIntroductionToChordalGraphsAndCliqueTrees__IVMA_1993}}]
\label{thm_cip}
The tree $T\in\SpanningTreesOf{\CliqueGraphOfG}$ is a clique tree, if and only if it fulfils the clique intersection property.
\end{theorem}

\begin{proof}
\par
The clique intersection property is a constraint only if $K_1\cap K_2\not=\emptyset$.
In this case, $\CFGeneratedBy{K_1\cap K_2}=: F$ is a finite clique family.
\par
Assume that $T\in\SpanningTreesOf{\CliqueGraphOfG}$ is a clique tree.
Thus, $T\InducedBy{F}$ is a subgraph of $\CliqueGraphOfG\InducedBy{F}$ and contains the unique $K_1$-$K_2$-path $P$ in $T$.
For every $K_3\in P$, we have $K_3\supsetneq\bigcap_{K'\in F} K' =\MaxGeneratorOf{F}=K_1\cap K_2$.
Thus, $T$ fulfils the clique intersection property.
\par
Assume that $T\in\SpanningTreesOf{\CliqueGraphOfG}$ fulfils the clique intersection property.
It implies that $T\InducedBy{F}$ must be a subgraph of $\CliqueGraphOfG\InducedBy{F}$.
By Proposition~\ref{prop_cliqueEdge_maxGenerator}, the set $\mathcal{C}:=\Set{K_1\cap K_2: K_1 K_2\in\CliqueGraphOfG}$ is the set of maximal generators of all finite clique families of cardinality at least two.
Therefore, $T\InducedBy{\CFGeneratedBy{C}}$ is a tree, for every $C\in\mathcal{C}$.
For the clique families $\CliquesOfG$ and $\Set{K}$, for each clique $K$, $T\InducedBy{F}$ is trivially a tree.
Conclude by~\eqref{eq_characterisation_inducedTree}.
\end{proof}
\par
An enumeration $\Set{K_1,K_2,\dotsc}$ of $\CliquesOfG$ has the \emph{running intersection property}~\cite[(3.1)]{Blair_Peyton__AnIntroductionToChordalGraphsAndCliqueTrees__IVMA_1993} (after~\cite[Condition 3.10]{Beeri_Fagin_Maier_Yannakis__OnTheDesirabilityOfAcyclicDatabaseSchemes__JACM_1983}), if
\begin{equation}\label{eq_rip}
 \forall\, 2\le n\in\NN:\exists\, 1\le i<n:
 \quad
 K_n\cap\bigcup_{j=1}^{n-1} K_j \subseteq K_i\,.
\end{equation}
A tree $T\in\SpanningTreesOf{\CliqueGraphOfG}$ has the \emph{running intersection property}, if there exists an enumeration of $\CliquesOfG$ with the running intersection property such that the $K_nK_i$ (with $i:=i(n)$ as in~\eqref{eq_rip}) are the edges of $T$.

\begin{theorem}[Generalisation of~{\cite[Theorem 3.4]{Blair_Peyton__AnIntroductionToChordalGraphsAndCliqueTrees__IVMA_1993}}]
\label{thm_rip}
The tree $T\in\SpanningTreesOf{\CliqueGraphOfG}$ is a clique tree, if and only if it has the running intersection property.
\end{theorem}

\begin{proof}
The proof of the finite case~\cite[Theorem 3.4]{Blair_Peyton__AnIntroductionToChordalGraphsAndCliqueTrees__IVMA_1993} shows the equivalence to the clique intersection property.
Thus, it generalises without modification to the infinite case.
\end{proof}
\par
For $T\in\CliqueTreesOfG$, one obtains an enumeration of $\CliquesOfG$ by fixing a root, starting with it, then enumerating all its children, then their children in turn and so on recursively.
\par
To a spanning tree $T\in\SpanningTreesOf{\CliqueGraphOfG}$ of the clique graph of a finite graph $G$ assign the weight $w(T):=\sum_{K_1 K_2\in T} \Cardinality{K_1\cap K_2}$.
The \emph{maximal weight spanning tree property} is another classic characterisation of finite clique trees.

\begin{theorem}[{\cite[Theorem 3.5]{Blair_Peyton__AnIntroductionToChordalGraphsAndCliqueTrees__IVMA_1993} after~\cite{Bernstein__PowerofNaturalSeminjoings__SIAMJComp_1981}}]
\label{thm_max_finite}
Let $G$ be a finite chordal graph.
The spanning tree $T\in\SpanningTreesOf{\CliqueGraphOfG}$ is a clique tree, if and only if $T$ has maximal weight with respect to $w$, that is
\begin{equation}\label{eq_max_finite}
 T\in\Argmax\Set{w(S) \mid S\in\SpanningTreesOf{\CliqueGraphOfG}}
 \,.
\end{equation}
\end{theorem}
\par
Condition~\eqref{eq_max_finite} makes no sense in the infinite case.
A local version holds, though.

\begin{theorem}
\label{thm_max_local}
Let $G$ be a chordal graph.
The spanning tree $T\in\SpanningTreesOf{\CliqueGraphOfG}$ is a clique tree, if and only if
\begin{equation}\label{eq_maximality_local}
 \forall F\in\CliqueFamiliesOfG, \Cardinality{F}<\infty:\quad
 T\InducedBy{F}\in
 \Argmax\Set{w(S) \mid
  S\in\SpanningTreesOf{\CliqueGraphOfG\InducedBy{F}}
 }\,.
\end{equation}
\end{theorem}

\begin{proof}
\par
We show the equivalence between~\eqref{eq_maximality_local} and~\eqref{eq_characterisation_contractedTree} by induction over the size of the maximal generator of a clique family.
The minimal clique families are $F(K)=\Set{K}$, for a clique $K$, and the equivalence holds trivially, as $\ContractionGraphOf{\Set{K}}$ contains only a single vertex $\Set{K}$ and no edges.
Suppose that $F$ has minimal cardinality and violates the equivalence.
Split the sum $w(T\InducedBy{F})$ into two parts.
The first part is a sum over edges in $\StrictSubfamilyGraphOfCF$.
By the minimality of $F$, the equivalence holds for all strict subfamilies of $F$ and this sum is a constant.
The second part is a sum over the edges in $\SameFamilyGraphOfCF$.
All edges in $\SameFamilyGraphOfCF$ have the same weight $\Cardinality{\MaxGeneratorOf{F}}$.
Hence, the equivalence between maximality of the second sum and the subgraph of $\ContractionGraphOfCF$ induced by the edge-labels of $T$ being spanning is obvious.
\end{proof}

\section*{Acknowledgements}
\par
The authors thank the anonymous referees for their detailed and constructive feedback on this work.
Christoph Hofer-Temmel acknowledges the support of the VIDI project ''Phase transitions, Euclidean fields and random fractals'', NWO 639.032.916.
Florian Lehner acknowledges the support of the Austrian Science Fund (FWF), project W1230-N13.

\bibliographystyle{plainnat}
\bibliography{ref}

\begin{thebibliography}{9}
\providecommand{\natexlab}[1]{#1}
\providecommand{\url}[1]{\texttt{#1}}
\expandafter\ifx\csname urlstyle\endcsname\relax
  \providecommand{\doi}[1]{doi: #1}\else
  \providecommand{\doi}{doi: \begingroup \urlstyle{rm}\Url}\fi

\bibitem[Beeri et~al.(1983)Beeri, Fagin, Maier, and
  Yannakakis]{Beeri_Fagin_Maier_Yannakis__OnTheDesirabilityOfAcyclicDatabaseSchemes__JACM_1983}
Catriel Beeri, Ronald Fagin, David Maier, and Mihalis Yannakakis.
\newblock On the desirability of acyclic database schemes.
\newblock \emph{J. Assoc. Comput. Mach.}, 30\penalty0 (3):\penalty0 479--513,
  1983.
\newblock ISSN 0004-5411.
\newblock \doi{10.1145/2402.322389}.
\newblock URL \url{http://dx.doi.org/10.1145/2402.322389}.

\bibitem[Bernstein and
  Goodman(1981)]{Bernstein__PowerofNaturalSeminjoings__SIAMJComp_1981}
Philip~A. Bernstein and Nathan Goodman.
\newblock Power of natural semijoins.
\newblock \emph{SIAM J. Comput.}, 10\penalty0 (4):\penalty0 751--771, 1981.
\newblock ISSN 0097-5397.
\newblock \doi{10.1137/0210059}.
\newblock URL \url{http://dx.doi.org/10.1137/0210059}.

\bibitem[Blair and
  Peyton(1993)]{Blair_Peyton__AnIntroductionToChordalGraphsAndCliqueTrees__IVMA_1993}
Jean R.~S. Blair and Barry Peyton.
\newblock An introduction to chordal graphs and clique trees.
\newblock In \emph{Graph theory and sparse matrix computation}, volume~56 of
  \emph{IMA Vol. Math. Appl.}, pages 1--29. Springer, New York, 1993.
\newblock \doi{10.1007/978-1-4613-8369-7_1}.
\newblock URL \url{http://dx.doi.org/10.1007/978-1-4613-8369-7_1}.

\bibitem[Diestel(2005)]{Diestel__GraphTheory__Springer_2005}
Reinhard Diestel.
\newblock \emph{Graph theory}, volume 173 of \emph{Graduate Texts in
  Mathematics}.
\newblock Springer-Verlag, Berlin, third edition, 2005.
\newblock ISBN 978-3-540-26182-7; 3-540-26182-6; 978-3-540-26183-4.

\bibitem[Dirac(1961)]{Dirac__OnRigidCircuitGraphs__Hamburg_1961}
G.~A. Dirac.
\newblock On rigid circuit graphs.
\newblock \emph{Abh. Math. Sem. Univ. Hamburg}, 25:\penalty0 71--76, 1961.
\newblock ISSN 0025-5858.

\bibitem[Galinier et~al.(1995)Galinier, Habib, and
  Paul]{Galinier_Habib_Paul__ChordalGraphsAndTheirCliqueGraphs__LNCS_1995}
Philippe Galinier, Michel Habib, and Christophe Paul.
\newblock Chordal graphs and their clique graphs.
\newblock In \emph{Graph-theoretic concepts in computer science ({A}achen,
  1995)}, volume 1017 of \emph{Lecture Notes in Comput. Sci.}, pages 358--371.
  Springer, Berlin, 1995.
\newblock \doi{10.1007/3-540-60618-1_88}.
\newblock URL \url{http://dx.doi.org/10.1007/3-540-60618-1_88}.

\bibitem[Gavril(1974)]{Gavril__TheIntersectionGraphsOfSubtreesInTreesAreExactlyTheChordalGraphs__JCTB_1974}
F{\u{a}}nic{\u{a}} Gavril.
\newblock The intersection graphs of subtrees in trees are exactly the chordal
  graphs.
\newblock \emph{J. Combinatorial Theory Ser. B}, 16:\penalty0 47--56, 1974.

\bibitem[Halin(1984)]{Halin__OnTheRepresentationOfTriangulatedGraphsInTrees__EJC_1984}
R.~Halin.
\newblock On the representation of triangulated graphs in trees.
\newblock \emph{European J. Combin.}, 5\penalty0 (1):\penalty0 23--28, 1984.
\newblock ISSN 0195-6698.

\bibitem[Ho and
  Lee(1989)]{Ho_Lee__CountingCliqueTreesAndComputingPerfectEliminationSchemesInParallel__IPL_1989}
Chin~Wen Ho and R.~C.~T. Lee.
\newblock Counting clique trees and computing perfect elimination schemes in
  parallel.
\newblock \emph{Inform. Process. Lett.}, 31\penalty0 (2):\penalty0 61--68,
  1989.
\newblock ISSN 0020-0190.
\newblock \doi{10.1016/0020-0190(89)90070-7}.
\newblock URL \url{http://dx.doi.org/10.1016/0020-0190(89)90070-7}.

\end{thebibliography}

\end{document}